\documentclass[11 pt]{amsart}

\usepackage{amsmath, amsthm,amssymb,bbm,enumerate,mathrsfs,mathtools}
\usepackage{a4wide}

\usepackage[font=small]{caption}
\usepackage{graphicx}
\usepackage{tikz,xcolor,verbatim}
\usetikzlibrary{calc,shapes}
 \usepackage[yyyymmdd,hhmmss]{datetime}
 \usepackage{background}
 \SetBgContents{ver. \currenttime \qquad \today }
 \SetBgScale{1}
 \SetBgAngle{0}
 \SetBgOpacity{1}
 \SetBgPosition{current page.south west} 
 \SetBgHshift{3cm}
 \SetBgVshift{0.5cm} 

\tikzset{
  bigblue/.style={circle, draw=blue!80,fill=blue!40,thick, inner sep=1.5pt, minimum size=5mm},
  bigred/.style={circle, draw=red!80,fill=red!40,thick, inner sep=1.5pt, minimum size=5mm},
  bigblack/.style={circle, draw=black!100,fill=black!40,thick, inner sep=1.5pt, minimum size=5mm},
  bluevertex/.style={circle, draw=blue!100,fill=blue!100,thick, inner sep=0pt, minimum size=2mm},
  redvertex/.style={circle, draw=red!100,fill=red!100,thick, inner sep=0pt, minimum size=2mm},
  blackvertex/.style={circle, draw=black!100,fill=black!100,thick, inner sep=0pt, minimum size=2mm},  
  whitevertex/.style={circle, draw=black!100,fill=white!100,thick, inner sep=0pt, minimum size=2mm},  
  smallblack/.style={circle, draw=black!100,fill=black!100,thick, inner sep=0pt, minimum size=1mm},
  smallwhite/.style={circle, draw=white!100,fill=white!100,thick, inner sep=0pt, minimum size=1mm} 
}
\usetikzlibrary{arrows,decorations.markings}


\usepackage{xcolor}

\title[The Pseudoforest analogue for the Strong Nine Dragon Tree Conjecture is True]{The Pseudoforest analogue for the Strong Nine Dragon Tree Conjecture is True}

\author[Grout]{Logan Grout}
\address[Logan Grout]{Departement of Combinatorics and Optimization, University of Waterloo, Waterloo, ON, Canada}
\email{lcgrout@uwaterloo.ca}

\author[Moore]{Benjamin Moore}
\address[Benjamin Moore]{Department of Combinatorics and Optimization, University of Waterloo, Waterloo, ON, Canada}  
\email{brmoore@uwaterloo.ca}

\thanks{
Both authors thank NSERC for financial support. 
}

\date{}

\newtheorem{thm}[equation]{Theorem}
\newtheorem{lemma}[equation]{Lemma}

\newtheorem{conj}[equation]{Conjecture}
\newtheorem{cor}[equation]{Corollary}
\newtheorem{claim}[equation]{Claim}

\theoremstyle{definition}
\newtheorem{definition}[equation]{Definition}
\newtheorem{obs}[equation]{Observation}

\newtheorem*{ack}{Acknowledgements}

\numberwithin{equation}{section}

\date{}

\begin{document}
\begin{abstract}
We prove that for any positive integers $k$ and $d$, if a graph $G$ has maximum average degree at most $2k + \frac{2d}{d+k+1}$, then $G$ decomposes into $k+1$ pseudoforests $C_{1},\ldots,C_{k+1}$ so that for at least one of the pseudoforests, each connected component has at most $d$ edges.
\end{abstract}

\maketitle
\section{Introduction}

For any graph $G$, a \textit{decomposition} of $G$ is a set of edge disjoint subgraphs of $G$ such that the union of their edges sets is the edge set of the graph.  Graph decompositions are a particularly heavily studied area of graph theory, and one of the most beautiful results about graph decompositions is the Nash-William characterization of when a graph decomposes into $k$ forests. 

\begin{thm}[\cite{Nashwilliams}, Nash-Williams Theorem]
A graph $G$ decomposes into $k$ forests if and only if 
\[ \max_{H \subseteq G, v(H) \geq 2} \frac{e(H)}{v(H)-1} \leq k.\]
\end{thm}

Here, $H \subseteq G$ is taken to mean $H$ is a subgraph of $G$. We will refer to $\max_{H \subseteq G, v(H) \geq 2} \frac{e(H)}{v(H)-1}$ as the \textit{fractional arboricity} of $G$.  It is not hard to see that the fractional arboricity of a graph might not be integral. For instance, the fractional arboricity of a cycle on $n$ vertices is $\frac{n}{n-1}$.  By the Nash-William characterization, this implies that cycles decompose into two forests, and you cannot decompose a cycle into a single forest.  Additionally, it is easy to see that cycles decompose into a forest and a matching, which is quite a bit stronger than just saying that cycles decompose into two forests. As the fractional arboricity of cycles is much closer to $1$ than to $2$, you might speculate that this is the reason you get this extra structure in the decomposition.  In general,  one might speculate that if the fractional arboricity of a graph is much closer to $k-1$ than to $k$, then not only does $G$ decompose into $k$ forests, but that one of the forests can be assumed to be a matching. Intuitively, this should be believable. If the fractional arboricity is very close to $k-1$ but still greater than $k-1$, then the Nash-William characterization says that you need $k$ forests, but just barely. Such a graph would roughly look like the union of $k-1$ forests, and then a few edges left over. It is reasonable to believe that you can force these left over edges to form a matching. Even more generally, you could conjecture that there exists an $\varepsilon  \in (0,1)$ such that if the fractional arboricity of $G$ is at most $k + \varepsilon$, then $G$ decomposes into $k+1$ forests, such that one of the forests has maximum degree at most $d$. The Nine Dragon Tree Conjecture (now proven by Hongbi Jiang and Daqing Yang \cite{ndt}), proposed by Mickeal Montassier, Patrice Ossona de Mendez, Andr\'e Raspaud, and Xuding Zhu in \cite{montassier} does just this. 

\begin{thm}[\cite{ndt}, Nine Dragon Tree Theorem]
Let $k$ and $d$ be positive integers. If the fractional arboricity of $G$ is at most $k + \frac{d}{k+d+1}$, then $G$ decomposes into $k+1$ forests such that one of the forests has maximum degree at most $d$. 
\end{thm}

In \cite{montassier}, it was shown that the fractional arboricity bound is sharp for arbitrarily large graphs. Despite this, Montassier et al. proposed a significant strengthening of the Nine Dragon Tree Theorem, aptly named the Strong Nine Dragon Tree Conjecture.

\begin{conj}[\cite{montassier}, Strong Nine Dragon Tree Conjecture]
\label{sndt}
Let $k$ and $d$ be positive integers. If the fractional arboricity of $G$ is at most $k + \frac{d}{k+d+1}$, then $G$ decomposes into $k+1$ forests such that one of the forests has each connected component containing at most $d$ edges. 
\end{conj}

This conjecture is wide open. The $d =1$ case was proven by Yang in \cite{Yangmatching}. The $k=1$ and $d=2$ case was proven by Kim, Kostochka, West, Wu, and Zhu \cite{kostochkaetal}.  Recently, the second author resolved the $d \leq k+1$ case \cite{moore2019approximate}. Moreover, it was shown that the conjecture is true if you replace ``$d$ edges" in the conclusion with a function $f(d,k)$ edges.

This paper will focus on pseudoforest decompositions. Recall, a \textit{pseudoforest} is a graph where each connected component contains at most one cycle. All of the above results or conjectures have pseudoforest analogues. The pseudoforest analogue of Nash-Williams Theorem is Hakimi's Theorem.

\begin{thm}[\cite{hakimithm}, Hakimi's Theorem]
\label{hakimistheorem}
A graph $G$ decomposes into $k$ pseudoforests if and only if the maximum average degree of $G$ is at most $2k$. 
\end{thm}

Here, the maximum average degree of a graph $G$ is
\[\text{mad}(G) := \max_{H \subseteq G} \frac{2e(H)}{v(H)}.\]

Genghua Fan, Yan Li, Nine Song, and Daqing Yang \cite{pseudoforest} proved the pseudoforest analogue of the Nine Dragon Tree Theorem.

\begin{thm}[\cite{pseudoforest}]
\label{pseudondt}
Let $k$ and $d$ be positive integers. If $\text{mad}(G) \leq 2k + \frac{2d}{d+k+1}$, then $G$ decomposes into $k+1$ pseudoforests, such that one of the pseudoforests has maximum degree at most $d$. 
\end{thm}

Fan et al. also showed the maximum average degree bound is best possible. 

\begin{thm}[\cite{pseudoforest}]
\label{sharpness}
For any positive integers $k$ and $d$, there are arbitrarily large graphs $G$ and an edge $e$ where $\text{mad}(G-e) = 2k + \frac{2d}{k+d+1}$, but $G$ does not decompose into $k+1$ pseudoforests where one of the pseudoforests has maximum degree $d$. 
\end{thm}

The main result of this paper is that the pseudoforest analogue of the Strong Nine Dragon Tree Conjecture is true.

\begin{thm}
\label{strongpndt}
Let $k$ and $d$ be positive integers. If $\text{mad}(G) \leq 2k + \frac{2d}{d+k+1}$, then $G$ decomposes into $k+1$ pseudoforests $T_{1},\ldots,T_{k},F$, such that each connected component of $F$ contains at most $d$ edges. 
\end{thm}

%

By Theorem \ref{sharpness}, the maximum average degree bound given in Theorem \ref{strongpndt} is sharp. Some of the cases of Theorem \ref{strongpndt} were known. As having maximum degree one is the same as having one of the pseudoforests be a matching, Theorem \ref{pseudondt} implies Theorem \ref{strongpndt} when $d=1$. Interestingly, the proof given in \cite{kostochkaetal} of the Strong Nine Dragon Tree Conjecture when $k = 1$ and $d=2$ implies Theorem \ref{strongpndt} when $k=1$ and $d=2$. Prior to our result, all other cases were open. 

As a template for how the proof of Theorem \ref{strongpndt} will proceed, we will give a proof of the non-trivial direction of Hakimi's Theorem. The proof we give appears in \cite{pseudoforest} (and is perhaps the first time it appeared in print, we are not aware of any earlier proofs), and the proof of the Nine Dragon Tree Theorem for pseudoforests in some sense follows this proof of Hakimi's Theorem. However, we feel that our proof more faithfully follows this proof of Hakimi's Theorem, and thus leads to the stronger result while having a shorter proof. 

Before proceeding, we need some definitions. Given a graph $G$, an \textit{orientation} of $G$ is obtained from $E(G)$ by taking each edge $xy$, and replacing $xy$ with exactly one of the arcs $(x,y)$ or $(y,x)$. To reverse the direction of an arc $(x,y)$ is to replace $(x,y)$ with the arc $(y,x)$. For any vertex $v$, let $d(v)$ denote the degree of $v$, and $d^{+}(v)$ denote the outdegree of $v$. A directed path $P$ from $u$ to $v$ is a path $P$ oriented so that $v$ is the only vertex with no outgoing edge.  The next observation is easy and well known.

\begin{obs}
A graph $G$ is a pseudoforest if and only if $G$ admits an orientation where every vertex has outdegree at most one. 
\end{obs}

From this observation, we get an important corollary.

\begin{cor}
\label{easycor}
A graph admits a decomposition into $k$ pseudoforests if and only if it admits an orientation such that every vertex has outdegree at most $k$.
\end{cor}

For a proof of Corollary \ref{easycor} we refer the reader to Corollary $1.2$ and Theorem $1.1$ of \cite{pseudoforest}. Alternatively, here is a short proof due to a referee. Given an orientation where each vertex has out degree at most $k$, colour the tails incident to each vertex with distinct colours from $1,\ldots,k$. Now each colour class of edges induces a subgraph with maximum outdegree at most one, hence it is a pseudoforest. For the converse, given $k$ pseudoforests, we orient each vertex to have maximum outdegree at most $1$. The union of these oriented pseudoforests is an orientation of $G$ with maximum outdegree at most $k$.  We will use Corollary \ref{easycor} repeatedly and implicitly throughout our proofs. With that, we can give a proof of Hakimi's Theorem.

\begin{proof}[Proof of Theorem \ref{hakimistheorem}]
We only prove that a graph with maximum average degree $2k$ decomposes into $k$ pseudoforests, as the other direction is trivial. 

Suppose towards a contradiction that $G$ has maximum average degree at most $2k$, but $G$ does not decompose into $k$ pseudoforests. Then $G$ does not admit an orientation such that each vertex has outdegree at most $k$. 

Consider an orientation $\vec{G}$ of $G$ that minimizes the sum
\[ \rho := \sum_{v \in V(G)} \max\{0, d^{+}(v) -k\}.\]
If this sum is zero, then we have a desired decomposition, a contradiction. Thus there is a vertex $v \in V(G)$ such that $v$ has outdegree at least $k+1$. If there is a directed path $P$ from $v$ to $x$ such that $x$ has outdegree at most $k-1$, then we can reverse the directions on all of the arcs on $P$ and obtain a decomposition with smaller $\rho$ value, a contradiction. Consider the subgraph $H$ induced by vertices which are reachable from directed paths of $v$. That is, $x \in V(H)$ if there is a directed path $P$ which starts at $v$ and ends at $x$. Then all vertices in $H$ have outdegree at least $k$, and $v$ has outdegree at least $k+1$. But this implies that the average degree of $H$ is strictly larger than $2k$, a contradiction. 
\end{proof}

Now we will give a high level overview of how our proof will proceed. We will take a pseudoforest decomposition $C_{1},\ldots,C_{t},F$ where we will try and bound the size of each connected component in $F$. In the above proof of Hakimi's Theorem, the bad situation was a vertex which had too large outdegree. Now, the bad situation is that there is a component which is too large. In the proof of Hakimi's Theorem, we searched for special paths to augment on from a vertex which had too large outdegree, and in our proof, we will search for paths to augment on from a component which is too large. In the proof of Hakimi's theorem, we 
identified a situation where we could augment our decomposition and obtain a better decomposition, namely, directed paths from a vertex with too large outdegree to a vertex with small outdegree. In our proof, we will identify similar situations, namely, finding two components which are small enough to augment our decomposition, or finding a large component which has at least two small components nearby to perform augmentations. Then we will show that when these configurations are removed, either we have a decomposition satisfying Theorem \ref{strongpndt} or our graph actually had too large maximum average degree to begin with. 

The paper is structured as follows. In Section \ref{mincounterexample} we  give the necessary definitions on how we will pick our counterexample, and prove basic properties about the counterexample. In Section \ref{flipping}, we describe how we will augment our decomposition in certain situations. In Section \ref{keyarguments}, we show how to use this augmentation strategy to either find an optimal decomposition or show that our graph has too large maximum average degree. 

Finally, to fix some notation we will let $v(G) = |V(G)|$ and $e(G) = |E(G)|$. All other undefined graph theory terminology can be found in \cite{BondyAndMurty}, or any other standard graph theory textbook.

\section{Picking the counterexample}
\label{mincounterexample}
In this section we describe how we will pick our counterexample. Fix $k$ and $d$ as positive integers, and suppose that $G$ is a vertex minimal counterexample to Theorem \ref{strongpndt} for the fixed values of $k$ and $d$.

Our first step will be to obtain desirable orientations of $G$. In particular, the orientations we will demand will imply that $G$ decomposes into $k$ pseudoforests each with $v(G)$ edges, and one left over pseudoforest. For this, we use a lemma proved in \cite{pseudoforest} (Lemma $2.1$). Technically, we need a stronger lemma, however the same proof as Lemma $2.1$ will suffice. We give a proof for completeness sake only, there is no new idea needed.  

\begin{lemma}[\cite{pseudoforest}]
\label{orientationlemma}
If $G$ is a vertex minimal counterexample to Theorem \ref{strongpndt}, then there exists an orientation of $G$ such that for all $v \in V(G)$, we have $k \leq d^{+}(v) \leq k+1$. 
\end{lemma}

\begin{proof}
Suppose no such orientation exists. As $G$ has maximum average degree at most $2k + 2$, by Hakimi's Theorem, $G$ admits an orientation so that every vertex has outdegree at most $k+1$. Orient $G$ so that every vertex has outdegree at most $k+1$, and that the sum 
\[\rho := \sum_{v \in V(G)} \max\{0, k - d^{+}(v)\}\]
is minimized. Observe that if $\rho$ is zero, then we have a desirable orientation.

First we claim there is a vertex $v$ with outdegree $k+1$. If not, then all vertices have outdegree at most $k$, and by Hakimi's Theorem $G$ decomposes into $k$ pseudoforests, contradicting that $G$ is a counterexample to Theorem \ref{strongpndt}. 

Now we claim there is no directed path $P$ from a vertex $v$ with outdegree $k+1$ to a vertex $u$ with outdegree at most $k-1$. Suppose towards a contradiction that $P$ is such a path. Then reversing the orientation on all of the arcs in $P$ gives a new orientation, where $v$ has outdegree $k$, all internal vertices have the same outdegree, and the outdegree of $u$ increases by one. But this contradicts that we picked our orientation to minimize $\rho$.

Let $S$ be the set of vertices in $G$ with out degree at most $k-1$, and let $S'$ be the set of vertices which have a directed path to a vertex in $S$. Observe that every vertex in $S'$ has outdegree at most $k$. Let $\bar{S'} = V(G) - S'$. Then every edge with one endpoint lying in $S'$ and one endpoint in $\bar{S'}$ is directed from $S'$ to $\bar{S'}$. Observe that $|\bar{S'}| < v(G)$. 

As $G$ is a vertex minimal counterexample we can decompose $G[\bar{S'}]$ into $k+1$ pseudoforests such that one of the pseudoforests has each connected component containing at most $d$ edges. Additionally, as every vertex in $S'$ has outdegree at most $k$, by Hakimi's Theorem we can decompose $G[S']$ into $k$ pseudoforests $C_{1},\ldots,C_{k}$. Thus we only need to deal with the edges between $\bar{S}'$ and $S'$. Observe that if $v$ has $t$ arcs $(v,u_{1}),\ldots,(v,u_{t})$ where $u_{i} \in \bar{S'}$ for all $i \in \{1,\ldots,t\}$, then $v$ has outdegree at most $k-t$ in $G[S']$. Thus $v$ has outdegree zero in at least $t$ of the pseudoforests $C_{1},\ldots,C_{k}$. Therefore we can add the arcs $(v,u_{1}),\ldots,(v,u_{t})$ to $t$ of the pseudoforests so that the result is a pseudoforest. As all arcs between $S'$ and $\bar{S}'$ are oriented from $S'$ to $\bar{S}'$, we now have a decomposition of $G$ which satisfies Theorem \ref{strongpndt}. But this contradicts that $G$ is a counterexample to Theorem \ref{strongpndt}. 
\end{proof}

 Let $\mathcal{F}$ be the set of orientations of $E(G)$ with $k \leq d^{+}(v) \leq k+1$ for each vertex $v \in V(G)$. A useful way of keeping track of our pseudoforest decomposition will be to colour the edges blue and red, where the edges coloured red will induce a pseudoforest. This will be the pseudoforest where we will want to bound the size of each connected component.  
 
\begin{definition}
Suppose $G$ is oriented such that $k \leq d^{+}(v) \leq k+1$ for each vertex $v \in V(G)$. Then a \textit{red-blue colouring of $G$} is a (non-proper)colouring of the edges where for any vertex $v \in V(G)$, we colour $k$ outgoing arcs of $v$ blue; if after this there is an uncoloured outgoing arc, colour this arc red.  
\end{definition}

Note that given an orientation in $\mathcal{F}$, one can generate many different red-blue colourings. As a graph decomposes into $k$ pseudoforests if and only if it admits an orientation where each vertex has outdegree at most $k$, we obtain the following observation. 

\begin{obs}
\label{redbluecolouringdecomp}
Given a red-blue colouring of $G$, we can decompose our graph $G$ into $k+1$ pseudoforests such that $k$ of the pseudoforests have all of their edges coloured blue, and the other pseudoforest has all of its edges coloured red. 
\end{obs}

Observe that one red-blue colouring can give rise to many different pseudoforest decompositions. Given a pseudoforest decomposition obtained from Observation \ref{redbluecolouringdecomp} we will say a pseudoforest which has all arcs coloured blue is a \textit{blue pseudoforest}, and the pseudoforest with all arcs coloured red is the \textit{red pseudoforest}.

\begin{definition}
Let $f$ be a red-blue colouring of $G$, and let $C_{1},\ldots,C_{k},F$ be a pseudoforest decomposition obtained from $f$ by Observation \ref{redbluecolouringdecomp}. Then we say that \textit{$C_{1},\ldots,C_{k},F$ is a pseudoforest decomposition generated from $f$}. We will always use the convention that $F$ is the red pseudoforest, and each $C_{i}$ is a blue pseudoforest. 
\end{definition}

 As $G$ is a counterexample, in every pseudoforest decomposition generated from a red-blue colouring, there is a component of the red pseudoforest which has more than $d$ edges. We define a \textit{residue function} which simply measures how close a decomposition is to satisfying Theorem \ref{strongpndt}.
 
\begin{definition}
Let $f$ be a red-blue colouring and $C_{1},\ldots,C_{k},F$ be a pseudoforest decomposition generated by $f$. Let $\mathcal{T}$ be the set of components of $F$. Then the \textit{residue function}, denoted $\rho$, is
\[\rho(F) = \sum_{K \in \mathcal{T}} \max\{e(K) -d,0\}.\]
\end{definition}

Using a red-blue colouring, and the resulting pseudoforest decomposition, we define an induced subgraph of $G$ which we will focus our attention on. Intuitively, this subgraph should be thought of as an ``exploration" subgraph similar to how in the proof of Hakimi's theorem we ``explored" from a vertex which had too large outdegree. Here we will ``explore" from a component which is too large. 

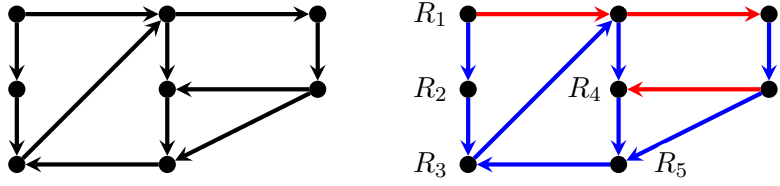
\begin{figure}
\begin{center}
\begin{tikzpicture}
\node[blackvertex] at (0,0) (v1) {};
\node[blackvertex] at (2,0) (v2) {};
\node[blackvertex] at (4,0) (v3) {};

\node[blackvertex] at (4,-1) (v4) {};
\node[blackvertex] at (2,-1) (v5) {};

\node[blackvertex] at (0,-1) (v6) {};

\node[blackvertex] at (0,-2) (v7) {};

\node[blackvertex] at (2,-2) (v8) {};

\draw[ultra thick,black, ->,>=stealth] (v1)--(v2);
\draw[ultra thick,black, ->, >=stealth] (v2)--(v3);
\draw[ultra thick,black, ->, >=stealth] (v1)--(v6);
\draw[ultra thick,black, ->, >= stealth] (v6)--(v7);
\draw[ultra thick,black, ->, >= stealth] (v4)--(v5);
\draw[ultra thick,black, ->, >= stealth] (v3)--(v4);
\draw[ultra thick,black, ->, >= stealth] (v2)--(v5);
\draw[ultra thick,black, ->, >= stealth] (v5)--(v8);
\draw[ultra thick,black,->, >= stealth] (v8)--(v7);
\draw[ultra thick,black,->, >= stealth] (v7)--(v2); 
\draw[ultra thick,black,->, >= stealth] (v4)--(v8);

\begin{scope}[xshift = 6cm]
\node[blackvertex] at (0,0) (v1) {};
\node[blackvertex] at (2,0) (v2) {};
\node[blackvertex] at (4,0) (v3) {};

\node[blackvertex] at (4,-1) (v4) {};
\node[blackvertex] at (2,-1) (v5) {};

\node[blackvertex] at (0,-1) (v6) {};

\node[blackvertex] at (0,-2) (v7) {};

\node[blackvertex] at (2,-2) (v8) {};

\draw[ultra thick,red, ->,>=stealth] (v1)--(v2);
\draw[ultra thick,red, ->, >=stealth] (v2)--(v3);
\draw[ultra thick,blue, ->, >=stealth] (v1)--(v6);
\draw[ultra thick,blue, ->, >= stealth] (v6)--(v7);
\draw[ultra thick,red, ->, >= stealth] (v4)--(v5);
\draw[ultra thick,blue, ->, >= stealth] (v3)--(v4);
\draw[ultra thick,blue, ->, >= stealth] (v2)--(v5);
\draw[ultra thick,blue, ->, >= stealth] (v5)--(v8);
\draw[ultra thick,blue,->, >= stealth] (v8)--(v7);
\draw[ultra thick,blue,->, >= stealth] (v7)--(v2); 
\draw[ultra thick,blue,->, >= stealth] (v4)--(v8);

\node[smallwhite] at (-.5,0) (v9) {$R_{1}$};
\node[smallwhite] at (-.5,-1) (v10) {$R_{2}$};
\node[smallwhite] at (-.5,-2) (v11) {$R_{3}$};
\node[smallwhite] at (2.7,-2) (v12) {$R_{5}$};
\node[smallwhite] at (1.55,-1) (v13) {$R_{4}$};
\end{scope}
\end{tikzpicture}
\end{center}
\caption{
\label{figure1}
In this example we assume $k=1$ and $d=1$. On the left, the orientation is in $\mathcal{F}$. On the right we have one possible red-blue colouring generated by this orientation. Here, the entire graph would be the exploration subgraph, and assuming $R_{1}$ is the root, $(R_{1},R_{2},R_{3},R_{4},R_{5})$ is the smallest legal order. Lastly, the isolated vertices are the small components (and are in fact the only possible small components when $k=1$ and $d=1$)}
\end{figure} 

\begin{definition}
Suppose that $f$ is a red-blue colouring of $G$, and suppose $D= (C_{1},\ldots,C_{k},F)$ is a pseudoforest decomposition generated from $f$. Let $R$ be a component of $F$ such that $e(R) >d$. We define the \textit{exploration subgraph} $H_{f,D,R}$ in the following manner. Let $S \subseteq V(G)$ where $v \in S$ if and only if there exists a path $P = v_{1},\ldots,v_{m}$ such that $v_{m}= v$, $v_{1} \in V(R)$, and either $v_{i}v_{i+1}$ is an arc $(v_{i},v_{i+1})$ coloured blue, or $v_{i}v_{i+1}$ is an arbitrarily directed arc coloured red. Then we let $H_{f,D,R}$ be the graph induced by $S$.
\end{definition}

Given a particular exploration subgraph $H_{f,D,R}$, we say $R$ is the \textit{root component}. We say the \textit{red components} of $H_{f,D,R}$ are the components of $F$ contained in $H_{f,D,R}$. 

It might not be clear why we made this particular definition for $H_{f,D,R}$, however the next observation shows that for any exploration subgraph $H_{f,D,R}$, the red edge density must be low. Before stating the observation, we fix some notation. Given a subgraph $K$ of $G$, we will let $E_{b}(K)$ and $E_{r}(K)$ denote the sets of edges of $K$ coloured blue and red, respectively. We let $e_{b}(K) = |E_{b}(K)|$ and $e_{r}(K) = |E_{r}(K)|$.

\begin{obs}
\label{finishingobservation}
For any red-blue colouring $f$, any pseudoforest decomposition $D$ generated from $f$, and any choice of root component $R$, the exploration subgraph $H_{f,D,R}$ satisfies
\[\frac{e_{r}(H_{f,D,R})}{v(H_{f,D,R})} \leq \frac{d}{d+k+1}.\]
\end{obs}

\begin{proof}

Suppose towards a contradiction that 
\[\frac{e_{r}(H_{f,D,R})}{v(H_{f,D,R})} > \frac{d}{d+k+1}.\]

As $H_{f,D,R}$ is an induced subgraph defined by directed paths, and every vertex $v \in V(G)$ has $k$ outgoing blue edges, each vertex in $H_{f,D,R}$ has $k$ outgoing blue edges. Thus,
\[\frac{e_{b}(H_{f,D,R})}{v(H_{f,D,R})} = k.\]

Then we have

\[\frac{\text{mad}(G)}{2} \geq \frac{e(H_{f,D,R})}{v(H_{f,D,R})} = \frac{e_{r}(H_{f,D,R})}{v(H_{f,D,R})} + \frac{e_{b}(H_{f,D,R})}{v(H_{f,D,R})} > k + \frac{d}{d+k+1}.\]

But this contradicts that $G$ has $\text{mad}(G) \leq 2k +\frac{2d}{d+k+1}$. 

\end{proof}

For the entire proof, we will be attempting to show that we can augment a given decomposition in such a way that either we obtain a decomposition satisfying Theorem \ref{strongpndt} or we can find a exploration subgraph $H_{f,D,R}$ which contradicts Observation \ref{finishingobservation}.

As Observation \ref{finishingobservation} allows us to focus only on red edges, it is natural to focus on red components which have small average degree. With this in mind, we define the notion of a \textit{small red component}.

\begin{definition}
Let $C_{1},\ldots,C_{t},F$ be a pseudoforest decomposition generated by a red-blue colouring. Let $K$ be a subgraph of $F$. Then $K$ is a \textit{small red subgraph} if    
\[\frac{e_{r}(K)}{v(K)} < \frac{d}{d+k+1}.\]
If $K$ is connected, we say \textit{$K$ is a small red component}
\end{definition}

In particular, we will be interested in the case when $K$ is connected and a small red component. When $K$ is connected the red subgraph is actually isomorphic to a tree, and we can rewrite the density bound in the definition in a more convenient manner. 

\begin{obs}
\label{AcyclicObservation}
Let $K$ be a connected small red subgraph. Then $K$ is acyclic, and further
\[e_{r}(K) < \frac{d}{k+1}.\]
\end{obs}

\begin{proof}
First, suppose that $K$ is not acyclic. Then $K$ contains exactly one cycle. As $K$ is connected, it follows that 
\[\frac{e_{r}(K)}{v(K)} = 1.\]
But $\frac{d}{k+d+1} < 1$ as $d$ and $k$ are positive integers, and hence $K$ would not be a small red subgraph. Thus we can assume that $K$ is acyclic, so $e(K) = v(K)-1$. Thus

\[\frac{e(K)}{v(K)} = \frac{e(K)}{e(K)+1} < \frac{d}{d+k+1}.\]
Therefore
\[e(K)(d+k+1) < d(e(K) + 1).\]
Simplifying, we see that this is equivalent to 
\[e(K) < \frac{d}{k+1}.\]
\end{proof}

We will want to augment our decomposition, and we will want a measure of progress that our decomposition is improving. Of course, if we reduce the residue function that clearly improves the decomposition. However, this might not always be possible, so we will introduce a notion of a ``legal order" of the red components. This order keeps track of the number of edges in components which are ``close" to the root component, with the idea being that if we can continually perform augmentations to make components ``closer" to the root component have fewer edges without creating any large components, then we eventually reduce the number of edges in the root component, which improves the residue function.   We formalize this in the following manner.  

\begin{definition}
  We call an ordering $(R_{1},\ldots,R_{t})$ of the red components of $H_{f,D,R}$ \textit{legal} if all red components are in the ordering, $R_{1}$ is the root component, and for all $j \in \{2,\ldots,t\}$ there exists an integer $i$ with $1 \leq i < j$ such that there is a blue arc $(u,v)$ such that $u \in V(R_{i})$ and $v \in V(R_{j})$. 
\end{definition} 

Let $(R_{1},\ldots,R_{t})$ be a legal ordering. We will say that $R_{i}$ is a \textit{parent} of $R_{j}$ if $i < j$ and there is a blue arc $(v_{i},v_{j})$ where $v_{i} \in R_{i}$  and $v_{j} \in R_{j}$. In this definition a red component may have many parents. To remedy this, if a red component has multiple parents, we arbitrarily pick one such red component and designate it as the only parent. If $R_{i}$ is the parent of $R_{j}$, then we say that $R_{j}$ is a \textit{child} of $R_{i}$. We say a red component $R_{i}$ is an \textit{ancestor} of $R_{j}$ if we can find a sequence of red components $R_{i_{1}},\ldots,R_{i_{m}}$ such that $R_{i_{1}} = R_{i}$, $R_{j_{m}} = R_{j}$, and $R_{i_{q}}$ is the parent of $R_{i_{q+1}}$ for all $q \in \{1,\ldots,m-1\}$. An important definition is that of vertices witnessing a legal order.
\begin{definition}
Given a legal order $(R_{1},\ldots,R_{t})$, we say a vertex $v$ \textit{witnesses the legal order for $R_{j}$} if there is a blue arc $(u,v)$ such that $u \in R_{i}$ and $v \in R_{j}$ and $i < j$. 
\end{definition}
  Observe that there may be many vertices which witness the legal order for a given red component. More importantly, for every component which is not the root, there exists a vertex which witnesses the legal order. We also want to compare two different legal orders. 

\begin{definition}
\label{defn}
Let $(R_{1},\ldots,R_{t})$ and $(R'_{1},\ldots,R'_{t'})$ be two legal orders. We will say $(R_{1},\ldots,R_{t})$ is \textit{smaller} than $(R'_{1},\ldots,R'_{t})$ if  the sequence $(e(R_{1}),\ldots,e(R_{t}))$ is smaller lexicographically than $(e(R'_{1}),\ldots,e(R'_{t'}))$. 
  
\end{definition}

With Definition \ref{defn}, we will pick our minimal counterexample $G$ in the following manner. 
First, $v(G)$ is minimized. 

After this we pick an orientation in $\mathcal{F}$, a red-blue colouring $f$ of this orientation, a pseudoforest decomposition $D =(C_{1},\ldots,C_{k},F)$ generated from $f$, such that the number of cycles in $F$ is minimized. Subject to this, we minimize the residue function $\rho$. Finally, we pick a smallest possible legal order $(R_{1},\ldots,R_{t})$. 

From here on out, we will assume we are working with a counterexample picked in the manner described. The point of minimizing the number of cycles in $F$ is slightly unintuitive compared to minimizing the residue function and minimizing the legal order. However, we minimize the number of cycles because when we augment we will need to ensure our decomposition is in fact a pseudoforest decomposition, and our augmentations will never create more cycles in $F$. Hence by minimizing the number of cycles in $F$ first, we can easily take care of the cases where cycles occur, which allows us to focus on the more important cases where the components are acyclic.

 It was pointed out to the second author via personal communication by Daqing Yang that one can remove the condition to minimize the number of cycles in $F$ by instead  modifying the definition of legal order to take into account the maximum average degree instead of just the number of edges. This is a nice approach and an argument could be made that this approach is more intuitive. However the proof ends up being fundamentally the same, and the authors prefer this approach.

\section{Augmenting the decomposition}
\label{flipping}

In this section we describe a very simple operation which will mostly be how we augment our decomposition. Let $f$ be the red-blue colouring of our counterexample, and let $C_{1},\ldots,C_{k}$ be the blue pseudoforests, and $F$ the red pseudoforest. Let $(R_{1},\ldots,R_{t})$ be the legal ordering picked for our counterexample. As some notation, given a vertex $x \in V(H_{f,D,R})$, we let $R^{x}$ denote the red component of $H_{f,D,R}$ containing $x$.  
\begin{definition}
Let $(x,y)$ be a blue arc. Further suppose that  $R^{y}$ is acyclic, and suppose that $e = xv$ is an arbitrarily oriented red arc incident to $x$. To \textit{exchange $e$ and $(x,y)$} is to perform the following procedure. First, take the maximal directed red path in $R^{y}$ starting at $y$, say $Q = y,v_{1},\ldots,v_{l}$ where $(v_{i},v_{i+1})$ is a red arc for $i \in \{1,\ldots,l\}$ and $(y,v_{1})$ is a red arc, and reverse the direction of all arcs of $Q$. Second, change the colour of $(x,y)$ to red and reorient $(x,y)$ to $(y,x)$. Finally, change the colour of $e$ to blue, and if $e$ is oriented $(v,x)$, reorient to $(x,v)$. 
\end{definition}

See Figure \ref{exchange} for an illustration. We note that exchanging on an edge $e$ and $(x,y)$ is well-defined. This is because $R^{y}$ is acyclic (and hence a tree), and thus there is a unique maximal directed path in $R^{y}$ which starts at $y$. 

\begin{obs}
\label{flipobservation}
Suppose we exchange the edge $e =xv$ and $(x,y)$. Then the resulting orientation is in $\mathcal{F}$, and the resulting colouring is a red-blue colouring of this orientation.  
\end{obs}

\begin{proof}

Let us first check the outdegrees of vertices after the exchange. Let $Q = y,v_{1},\ldots,v_{l}$ be the maximal directed red path in $R^{y}$ before exchanging the edges. First suppose that $Q$ is not just $y$. Then all of the internal vertices on this path have the same outdegree after reversing as before. On the other hand, the outdegree of $y$ decreases by $1$, and the outdegree of $v_{l}$ increases by one. As $Q$ is a maximal red path, this implies that the outdegree of $v_{l}$ before reversing the arcs on $Q$ was $k$, and hence after reversing the arcs this outdegree is $k+1$. The outdegree of $y$ drops by one after reversing the arcs on $Q$, but we reverse the arc $(x,y)$ to $(y,x)$, and hence the outdegree of $y$ remains the same as before. 

If $Q$ is just $y$, then the outdegree of $y$ before exchanging was $k$, and then as we reorient $(x,y)$ to $(y,x)$, the outdegree of $y$ is now $k+1$. 

Focusing on $x$ now, if $e$ is oriented from $x$ to $v$, then the outdegree of $x$ is initially $k+1$, and after the exchange it ends up being $k$; otherwise, we reorient $(v,x)$ to $(x,v)$ and so the outdegree of $x$ remains the same as before the exchange. Lastly, the outdegree of $v$ remains the same if $e$ was oriented $(x,v)$, and otherwise the outdegree of $v$ initially was $k+1$, and after reorienting becomes $k$. Thus the resulting orientation is in $\mathcal{F}$. 

Now we will see that this new colouring is a red-blue colouring of the orientation. Note after exchanging $e$ and $(x,y)$, $y$ has exactly one outgoing edge coloured red. If $x$ had no outgoing red edge before, it still has no outgoing red edge, and if it did have an outgoing red edge, then the outdegree of $x$ dropped by one, and now $x$ has no outgoing red edge. Finally, if $Q$ was not just $y$, then $v_{l}$ now has no outgoing red edge. It follows that the resulting colouring is in fact a red-blue colouring. 
\end{proof}

\begin{figure}

\begin{tikzpicture}
\node[blackvertex] at (0,0) (v1) {};
\node[smallwhite] at (-.3,0) (v5) {$x$};
\node[blackvertex] at (1.5,0) (v2) {};

\draw[ultra thick, ->, red] (v1)--(v2) {};

\node[blackvertex] at (0,-1.25) (v3) {};
\node[smallwhite] at (-.3,-1.25) (v6) {$y$};
\node[blackvertex] at (1.5,-1.25) (v4) {};
\draw[ultra thick, ->, blue] (v1)--(v3);
\draw[ultra thick, <-, red] (v4) --(v3);
\node[smallwhite] at (.75, .3) (v6) {$e$};

\begin{scope}[xshift = 3cm]

\node[blackvertex] at (0,0) (v1) {};
\node[smallwhite] at (-.3,0) (v5) {$x$};
\node[blackvertex] at (1.5,0) (v2) {};

\draw[ultra thick, ->, blue] (v1)--(v2) {};

\node[blackvertex] at (0,-1.25) (v3) {};
\node[smallwhite] at (-.3,-1.25) (v6) {$y$};
\node[blackvertex] at (1.5,-1.25) (v4) {};
\draw[ultra thick, <-, red] (v1)--(v3);
\draw[ultra thick, ->, red] (v4) --(v3);
\node[smallwhite] at (.75, .3) (v6) {$e$};
\end{scope}
\end{tikzpicture}

\caption{
\protect\label{exchange}
 An example of an exchange on an edge $e$ and $(x,y)$
}

\end{figure}
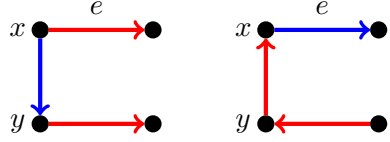

To avoid repetitively mentioning it, we will implicitly make use of Observation \ref{flipobservation}. Now we begin to impose some structure on our decomposition. First we make an observation which allows us to effectively ignore parent components with cycles (such components will still exist, but for the purposes of our argument we will not need to worry about them). 

\begin{obs}
\label{redcyclesaturationlemma}
Let $(x,y)$ be a blue arc such that $R^{x}$ is distinct from $R^{y}$ and further $R^{y}$ is a tree. Then $x$ does not lie in a cycle of $F$. 
\end{obs}

\begin{proof}
Suppose towards a contradiction that $x$ lies in a cycle of $F$. Let $e$ be an edge incident to $x$ which lies in the cycle coloured red. Now exchange $(x,y)$ and $e$. As $(x,y)$ was an arc between two distinct red components, and $e$ was in the cycle coloured red, after performing the exchange, we reduce the number of cycles in $F$ by one. However, this contradicts that we picked our counterexample to have the fewest number of red cycles.  
\end{proof}

With this we can show that given two components $K$ and $C$, where $K$ is the parent of $C$, and $C$ is acyclic, that $e_{r}(K) + e_{r}(C) \geq d$.

\begin{lemma}
\label{troublesomedonttroublesome}
Let $R^{x}$ and $R^{y}$ be red components such that $R^{y}$ is the child of $R^{x}$, $R^{y}$ does not contain a cycle, and $(x,y)$ is a blue arc from $x$ to $y$. Then $e_{r}(R^{x}) + e_{r}(R^{y}) \geq d$. 
\end{lemma}

\begin{proof}
Suppose towards a contradiction that $e_{r}(R^{x}) + e_{r}(R^{y}) < d$. Hence $e_{r}(R^{x}) < d$. Thus $R^{x}$ is not the root.  Let $w$ be a vertex which witnesses the legal order for $R^{x}$ ($w$ exists as $R^{x}$ is not the root). By Observation \ref{redcyclesaturationlemma} we know that $x$ does not belong to a cycle of $R^{x}$.

\textbf{Case 1:} $w \neq x$.

Let $e$ be the edge incident to $x$ in $R^{x}$ such that $e$ lies on the path from $x$ to $w$ in $R^{x}$. Then exchange $(x,y)$ and $e$. As $e_{r}(R^{x}) + e_{r}(R^{y}) < d$, all resulting red components have fewer than $d$ edges, and hence we do not increase the residue function. Furthermore, we claim we can find a smaller legal order. Let $R_{i}$ be the component in the legal order corresponding to $R^{x}$. Then consider the new legal order where the components $R_{1},\ldots,R_{i-1}$ remain in the same position, we replace $R_{i}$ with the new red component containing $w$, and then complete the order arbitrarily. By how we picked $e$, $e(R^{w})$ is strictly smaller than $e(R_{i})$, and hence we have found a smaller legal order, a contradiction. 

\textbf{Case 2:} $w =x$.

We refer the reader to Figure \ref{Case2} for an illustration. 
As $R^{x}$ is not the root component, let $R^{x_{1}}$ be the closest ancestor of $R^{x}$ such that $e(R^{x_{1}}) \geq 1$ (there is an ancestor with this property, as the root has at least one edge). Let $R^{x_{n}},R^{x_{n-1}},\ldots,R^{x_{1}}$ be a sequence of red components such that for $i \in \{2,\ldots,n\}$, $R^{x_{i}}$ is the child of $R^{x_{i-1}}$ and $R^{x}$ is the child of $R^{x_{n}}$. Up to relabelling the vertices,  there is a path $P = x_{1},\ldots,x_{n},x,y$ such that $(x_{i},x_{i+1})$ is an arc coloured blue, and $(x_{n},x)$ is an arc coloured blue. Let $e$ be a red edge incident to $x_{n}$. Now do the following. Colour $(x,y)$ red, and reverse the direction of all arcs in $P$. Colour $e$ blue, and orient $e$ away from $x_{1}$.  By the argument in our proof of  Observation \ref{flipobservation}, the resulting orientation is in $\mathcal{F}$ and the colouring described is a red-blue colouring. Furthermore as $e_{r}(R^{x}) + e_{r}(R^{y}) < d$, all resulting red components have at most $d$ edges, and hence the residue function did not increase (in the event that $R^{x_{1}}$ is the root, the residue function strictly decreases, so we assume that $R^{x_{1}}$ is not the root). Finally, we can find a smaller legal order in this orientation, as we simply take the same legal order up to the component containing $x_{1}$, and then complete the remaining order arbitrarily. As the component containing $x_{1}$ has at least one fewer edge now, this order is a smaller legal order, a contradiction.
\end{proof}

\begin{figure}
\begin{tikzpicture}

\node[blackvertex] at (0,0) (v1) {};
\node[blackvertex] at (2,0) (v2) {};
\node[blackvertex] at (4,0) (v3) {};

\node[blackvertex] at (4,-1) (v4) {};
\node[blackvertex] at (2,-1) (v5) {};

\node[blackvertex] at (0,-1) (v6) {};

\node[blackvertex] at (0,-2) (v7) {};

\node[blackvertex] at (2,-2) (v8) {};

\draw[ultra thick,red, ->,>=stealth] (v1)--(v2);
\draw[ultra thick,red, ->, >=stealth] (v2)--(v3);
\draw[ultra thick,blue, ->, >=stealth] (v1)--(v6);
\draw[ultra thick,blue, ->, >= stealth] (v6)--(v7);
\draw[ultra thick,red, ->, >= stealth] (v4)--(v5);
\draw[ultra thick,blue, ->, >= stealth] (v3)--(v4);
\draw[ultra thick,blue, ->, >= stealth] (v2)--(v5);
\draw[ultra thick,blue, ->, >= stealth] (v5)--(v8);
\draw[ultra thick,blue,->, >= stealth] (v8)--(v7);
\draw[ultra thick,blue,->, >= stealth] (v7)--(v2); 
\draw[ultra thick,blue,->, >= stealth] (v4)--(v8);

\node[smallwhite] at (-.5,0) (v9) {$R_{1}$};
\node[smallwhite] at (-1,-1) (v10) {$R_{2}$};
\node[smallwhite] at (-.5,-1) (v14) {$x$};
\node[smallwhite] at(-.5,-2) (v15) {$y$};
\node[smallwhite] at (-1,-2) (v11) {$R_{3}$};
\node[smallwhite] at (2.7,-2) (v12) {$R_{5}$};
\node[smallwhite] at (1.55,-1) (v13) {$R_{4}$};

\begin{scope}[xshift = 6cm]
\node[blackvertex] at (0,0) (v1) {};
\node[blackvertex] at (2,0) (v2) {};
\node[blackvertex] at (4,0) (v3) {};

\node[blackvertex] at (4,-1) (v4) {};
\node[blackvertex] at (2,-1) (v5) {};

\node[blackvertex] at (0,-1) (v6) {};

\node[blackvertex] at (0,-2) (v7) {};

\node[blackvertex] at (2,-2) (v8) {};

\draw[ultra thick,blue, ->,>=stealth] (v1)--(v2);
\draw[ultra thick,red, ->, >=stealth] (v2)--(v3);
\draw[ultra thick,blue, <-, >=stealth] (v1)--(v6);
\draw[ultra thick,red, <-, >= stealth] (v6)--(v7);
\draw[ultra thick,red, ->, >= stealth] (v4)--(v5);
\draw[ultra thick,blue, ->, >= stealth] (v3)--(v4);
\draw[ultra thick,blue, ->, >= stealth] (v2)--(v5);
\draw[ultra thick,blue, ->, >= stealth] (v5)--(v8);
\draw[ultra thick,blue,->, >= stealth] (v8)--(v7);
\draw[ultra thick,blue,->, >= stealth] (v7)--(v2); 
\draw[ultra thick,blue,->, >= stealth] (v4)--(v8);

\end{scope}
\end{tikzpicture}
\caption{
\protect\label{Case2}
 An illustration of Case $2$ in Lemma \ref{troublesomedonttroublesome}, where $k=1$ and $d=1$ and $x =w$. A vertex in the root component was the ancestor,  so in this case we reduce the residue function. }
\end{figure}
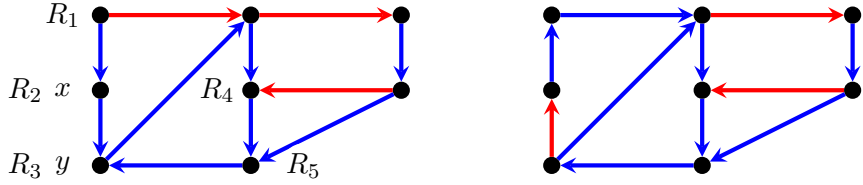

We note the following important special case of Lemma \ref{troublesomedonttroublesome}, that small red components do not have small red children. 

\begin{cor}
\label{smallchildrencorollary}
If $K$ is a small red component, then $K$ does not have any small red children. 
\end{cor}

\begin{proof}
Suppose towards a contradiction that $K$ has a small red child $C$. As $K$ is small, then $e_{r}(K) < \frac{d}{k+1}$. Similarly, $e_{r}(C) < \frac{d}{k+1}$. But then $e_{r}(K) + e_{r}(C) < \frac{2d}{k+1} \leq d$, contradicting Lemma \ref{troublesomedonttroublesome}.
\end{proof}

Now we will show that every red component has at most $k$ small red children. 

\begin{lemma}
\label{boundingthenumberofsmallchildren}
If $K$ is a red component, then $K$ has at most $k$ small children. 
\end{lemma}

\begin{proof}
Suppose towards a contradiction that $K$ has at least $k+1$ distinct small children. Then by the pigeon-hole principle, there are two distinct small children $C_{1}$ and $C_{2}$ such that there are blue arcs $(x,x')$, $(y,y')$ so that $x \neq y$, $x,y \in V(K)$, $x' \in V(C_{1})$ and $y' \in V(C_{2})$. By Observation \ref{redcyclesaturationlemma} we can assume that neither $x$ nor $y$ lies in a red cycle in $K$. Consider a path $P_{x,y}$ in $K$ from $x$ to $y$ (not directed). Let $e_{x}$ be the edge incident to $x$ in $P_{x,y}$ and $e_{y}$ be the edge incident to $y$ in $P_{x,y}$. Let $K_{x}$ denote the component of $K-e_{y}$ which contains $x$, and let $K_{y}$ denote the component of $K-e_{x}$ which contains $y$. 

\begin{claim}
\label{aclaim}
$e_{r}(K_{x}) \leq e_{r}(T_{2})$ and $e_{r}(K_{y}) \leq e_{r}(T_{1})$. 
\end{claim}

\begin{proof}
By symmetry, we will only show that $e_{r}(K_{x}) \leq e_{r}(T_{2})$. So suppose towards a contradiction that $e_{r}(K_{x}) > e_{r}(T_{2})$. Then exchange on $(y,y')$ and $e_{y}$. As $e_{r}(K_{x}) > e(T_{2})$, the residue function does not increase. Observe that if $K$ is the root, then the residue function will decrease, and that will give a contradiction. Thus we may assume that $K$ is not the root. We claim we can find a smaller legal order. If there is a vertex which witnesses the legal order for $K$ in $K_{x}$, then taking the same legal order up to $K$ and then replacing $K$ with $K_{x}$ gives a smaller legal order. Similarly, if there is no vertex which witnesses the legal order in $K_{x}$, then because $e_{r}(K_{x}) > e_{r}(T_{2})$, taking the same legal order up to $K$ and replacing $K$ with the component containing $y$ of $K-e_{y}$ after the exchange, and filling in the rest of the order arbitrarily gives a smaller legal order. In both cases, this is a contradiction. 
\end{proof}

Note that either $e_{r}(K) \leq e_{r}(K_{x}) + e_{r}(K_{y})$ or $e_{r}(K) \leq e_{r}(K_{x}) + e_{r}(K_{y}) -1$ (the first case occurs if $e_{y} \neq e_{x}$, and the second occurs if  $e_{y} = e_{x}$). Since each $T_{i}$ is a small child, Claim \ref{aclaim} (together with Observation \ref{AcyclicObservation}) implies that
\[e_{r}(K_{x}) + e_{r}(K_{y}) < \frac{d}{k+1} + \frac{d}{k+1} \leq d.\]

Hence, $e_{r}(K) \leq d$.  Thus we can assume that $K$ is not the root. Let $w$ be a vertex which witnesses the legal order. Without loss of generality, we can assume that $w \in V(K_{x})$. Then exchange on $(y,y')$ and $e_{y}$. We do not increase the residue function as $e_{r}(K_{y}) \leq e_{r}(T_{1}) < \frac{d}{k+1}$. However, we can find a smaller legal order by taking the same legal order up to $K$, and replacing $K$ with $K_{x}$, and completing this order arbitrarily. But this contradicts our choice of legal order, a contradiction. 
\end{proof}

We are now in position to prove the theorem. 

\section{Bounding the maximum average degree}
\label{keyarguments}

In this section, we give a counting argument to show that our chosen exploration subgraph has too large average degree. We make the following definition for ease of notation.

\begin{definition}
Let $K$ be a red component, and let  $K_{1},\ldots,K_{q}$ be the small red children of $K$. We will let $K_{C}$ denote the subgraph with vertex set $V(K_{C}) = V(K) \cup V(K_{1}) \cup \cdots \cup V(K_{q})$, that contains all red edges from $K,K_{1},\ldots,K_{q}$. 
\end{definition}

\begin{lemma}
\label{counting}
Let $K$ be a red component which is not small. Then $K_{C}$ is not small. Further, if $e_{r}(K) > d$, then 
\[\frac{e_{r}(K_{C})}{v(K_{C})} > \frac{d}{d+k+1}.\]
\end{lemma}

\begin{proof}
First, observe that if $K$ has no small children then $K_{C} = K$ and hence is not small. If $e_{r}(K) >d$, then as $K$ is connected, $v(K) \leq e_{r}(K) +1$ and hence $e_{r}(K)/v(K) \geq (d+1)/(d+2) > d/(d+k+1)$. Thus we can suppose that $K$ has small children  $K_{1},\ldots,K_{q}$. By Lemma \ref{boundingthenumberofsmallchildren}, $q \leq k$.  By Lemma \ref{troublesomedonttroublesome}, we know for every $i \in \{1,\ldots,q\}$ the inequality $e_{r}(K) + e_{r}(K_{i}) \geq d$ holds. As $e_{r}(K_{i}) \geq 0$ for all $i \in \{1,\ldots,q\}$, it follows that $e_{r}(K_{i}) \geq \max\{0,d- e(K)\}$ for all $i \in \{1,\ldots,q\}$. 

Then a quick calculation shows
\begin{align*}
\frac{e_{r}(K_{C})}{v(K_{C})} &= \frac{e_{r}(K) +\sum_{i =1}^{q} e_{r}(K_{i})}{v(K) + \sum_{i =1}^{q} v(K_{i})}\\
 & \geq \frac{e_{r}(K) +\sum_{i =1}^{q} \max\{0,d-e_{r}(K)\}}{e_{r}(K) + 1 + q + \sum_{i=1}^{q} \max\{0,d-e_{r}(K)\}}
 \\ & \geq \frac{e_{r}(K) +\sum_{i =1}^{q} \max\{0,d-e_{r}(K)\}}{e_{r}(K) + 1 + k + \sum_{i=1}^{q} \max\{0,d-e_{r}(K)\}}.
\end{align*}

The first equality is simply applying the definition of $K_{C}$. The first inequality uses that $e_{r}(K_{i}) \geq \max\{0,d-e_{r}(K)\}$, and that as $K_{i}$ is small, $K_{i}$ is a tree, so $v(K_{i}) = e_{r}(K_{i}) +1$. Finally, the second inequality is using that $q \leq k$. Now we split this into two cases based on whether or not $\max\{0,d-e_{r}(K)\}$ is $0$ or $d-e_{r}(K)$.

\textbf{Case 1:} $\max\{0,d-e_{r}(K)\} = 0$.

If $\max\{0,d-e_{r}(K)\} =0$, then $e_{r}(K) \geq d$. Thus it follows that, 
\begin{align*}
\frac{e_{r}(K) +\sum_{i =1}^{q} \max\{0,d-e_{r}(K)\}}{e_{r}(K) + 1 + k + \sum_{i=1}^{q} \max\{0,d-e_{r}(K)\}} &= \frac{e_{r}(K)}{e_{r}(K) + k+1} \\ 
&\geq \frac{d}{d+k+1}.
\end{align*}
Further, if $e_{r}(K) >d$, the above inequality is strict.  

\textbf{Case 2:} $\max\{0,d-e_{r}(K)\} = d-e_{r}(K)$.
As $e_{r}(K) \leq d$, we only need to show that $K$ is not small. Calculating we obtain,
\begin{align*}
\frac{e_{r}(K) +\sum_{i =1}^{q} \max\{0,d-e_{r}(K)\}}{e_{r}(K) + 1 + k + \sum_{i=1}^{q} \max\{0,d-e_{r}(K)\}} &=
\frac{e_{r}(K) +q(d - e_{r}(K))}{e_{r}(K) + q(d-e_{r}(K)) +k+1} \\
&\geq \frac{d - e_{r}(K) +e_{r}(K)}{e_{r}(K) +d - e_{r}(K) +k+1} \\
&= \frac{d}{d+k+1}.
\end{align*}
\end{proof}

Now we finish the proof.  Let $\mathcal{R}$ denote the set of red components of $H_{f,D,R}$ which are not small. By Corollary \ref{smallchildrencorollary} it follows that,
\[V(H_{f,D,R}) = \bigcup_{K \in \mathcal{R}} V(K_{C}).\]

This follows since a small component cannot have a small child. Therefore it follows that:
\[E_{r}(H_{f,D,R}) = \bigcup_{K \in \mathcal{R}} E(K_{C}).\]

Now we bound the maximum average degree of $H_{f,D,R}$. By Lemma \ref{counting}, we have
\begin{align*}
\frac{e_{r}(H_{f,D,R})}{v(H_{f,D,R})} &= \frac{\sum_{K \in \mathcal{R}}e_{r}(K_{C})}{\sum_{K \in \mathcal{R}} v(K_{C})} \\
&> \frac{d}{d+k+1}.
\end{align*}

Here, equality holds in the first line because we chose unique parents for components. The strict inequality follows as $K_{C}$ is not small for any $K \in \mathcal{R}$ by Lemma \ref{counting}, and further the root component satisfies $e(R) >d$. However, this contradicts Observation \ref{finishingobservation}. Theorem \ref{strongpndt} follows. 

\begin{ack}
The authors are extremely thankful to Joseph Cheriyan for numerous helpful discussions about the problem. We thank the anonymous referees for helpful comments in improving the presentation.
\end{ack}

\bibliographystyle{plain}
\bibliography{pseudo}

\end{document}